\documentclass[a4paper]{article}

\usepackage[english]{babel}
\usepackage[utf8x]{inputenc}
\usepackage[T1]{fontenc}
\usepackage{latexsym,amssymb,amsmath,amsfonts,amscd} 

\usepackage{amsmath}
\usepackage{graphicx}
\usepackage[colorinlistoftodos]{todonotes}
\usepackage[colorlinks=true, allcolors=blue]{hyperref}

\newtheorem{thm}{theorem}[section]
\newtheorem{theorem}[thm]{Theorem}
\newtheorem{proposition}[thm]{Proposition}
\newtheorem{lema}[thm]{Lemma}
\newtheorem{corollary}[thm]{Corollary}
\newtheorem{remark}[thm]{Remark}
\newtheorem{example}[thm]{Example}
\newtheorem{definition}[thm]{Definition}

\newenvironment{proof}[1][Proof]{\noindent\textbf{#1\,} }{\hfill \rule{0.5em}{0.5em}\medskip}

\title{$\mathbb{Z}$-gradings  of full support on the Grassmann algebra}
\author{Alan Guimar\~aes \\
	Departamento de Matemática\\
	Universidade Federal do Rio Grande do Norte\\
	Natal, RN, 59078-970, Brazil\\
	alansimoes10@hotmail.com\\
	\\
Antonio Brandão Jr.\\
Unidade Acadêmica de Matemática\\
Universidade Federal de Campina Grande\\
Campina Grande, PB 58429-970, Brazil\\
brandao@mat.ufcg.edu.br\\
	\\
	Claudemir Fidelis\thanks{ {\bf Corresponding Author}}~\thanks{Supported by FAPESP grant No.~2019/12498-0}\\ 
	Unidade Acadêmica de Matemática\\
	Universidade Federal de Campina Grande\\
	Campina Grande, PB 58429-970, Brazil
	\\ and \\Instituto de Matem\'atica e Estat\'istica da Universidade de S\~ao Paulo\\
	S\~ao Paulo, SP, 05508-090, Brazil\\
	claudemir@mat.ufcg.edu.br\\
} 

\date{}
\begin{document}
\maketitle

\begin{abstract}
Let $E$ be the infinite dimensional Grassmann algebra over a field $F$ of characteristic zero. In this paper we investigate the structures of $\mathbb{Z}$-gradings on $E$ of full support. Using methods of elementary number theory, we describe the $\mathbb{Z}$-graded polynomial identities for the so-called $2$-induced $\mathbb{Z}$-gradings on $E$ of full support. As a consequence of this fact we provide examples of $\mathbb{Z}$-gradings on $E$ which are PI-equivalent but not $\mathbb{Z}$-isomorphic. This is the first example of graded algebras with infinite support that are PI-equivalent and not isomorphic as graded algebras. We also present the notion of central $\mathbb{Z}$-gradings on $E$ and we show that its $\mathbb{Z}$-graded polynomial identities are closely related to the $\mathbb{Z}_{2}$-graded polynomial identities of $\mathbb{Z}_{2}$-gradings on $E$.
 
 
 \medskip
 
\end{abstract}

\section{Introduction}

Let $F$ be a field of characteristic zero and let $L$ be a vector space over $F$ with basis $e_{1}$, $e_{2}$, \dots. The infinite dimensional Grassmann algebra $E$ of $L$ over $F$ is the vector space with a basis consisting of 1 and all products $e_{i_1}e_{i_2}\cdots e_{i_k}$ where $i_{1}<i_{2}<\ldots <i_{k}$, $k\geq 1$. The multiplication in $E$ is induced by $e_{i}e_{j}=-e_{j}e_{i}$ for all $i$ and $j$. We shall denote by $B=B_{E}$ the above canonical basis of $E$. The Grassmann algebra has a natural $\mathbb{Z}_2$-grading $E_{can}=E_{(0)}\oplus E_{(1)}$ where $E_{(0)}$ is the vector space spanned by 1 and all products $e_{i_1}\cdots e_{i_k}$ with even $k$ while $E_{(1)}$ is the vector space spanned by the products with odd $k$. Clearly $E_{(0)}$ is just the centre of $E$ and $E_{(1)}$ is the ``anticommuting'' part of $E$. This natural structure of a $\mathbb{Z}_2$-graded algebra on $E$ makes the Grassmann algebra very important not only in Algebra but in various areas of Mathematics as well as in Theoretical Physics. We are not going to discuss the applications of the Grassmann algebra; we point out to the book by Berezin \cite{berezin} for such applications.

The Grassmann algebra $E$ is one of the most important algebras satisfying a polynomial identity, also known as PI-algebras. 
Kemer \cite{basekemer2, basekemer} developed the structure theory of the ideals of identities of associative algebras, also called T-ideals, in characteristic 0. This description led Kemer to the positive solution of the long standing and extremely important problem proposed by W. Specht in 1950: is every T-ideal in characteristic 0 finitely generated as a T-ideal? A key ingredient in Kemer's research was the following result: every associative PI-algebra over a field of characteristic zero is PI-equivalent to the Grassmann envelope of a finite dimensional associative superalgebra.

Recall that in characteristic 2 the Grassmann algebra is commutative hence not very interesting from the point of view of PI theory. The polynomial identities of the Grassmann algebra were described by Latyshev \cite{latyshev}, and later on by Krakowski and Regev \cite{KR}. The identities of $E$ have been extensively studied also in positive characteristic. The interested reader can consult the paper \cite{agpk} and the references therein for further information. In characteristic 0, it is the easiest example of an algebra which does not satisfy any standard polynomial. We recall that the standard polynomial is the polynomial $s_n(x_1,\ldots, x_n) = \sum_{\sigma\in S_n} (-1)^\sigma x_{\sigma(1)}\cdots x_{\sigma(n)}$, where $S_n$ stands for the symmetric group on the letters $\{1,\ldots, n\}$, and $(-1)^\sigma$ is the sign of the permutation $\sigma$. This polynomial plays a prominent role in PI theory, especially in characteristic 0.

Group gradings on algebras and the corresponding graded polynomial identities have been extensively studied in PI theory during the last two or three decades. The graded identities for the natural (or canonical) grading on $E$ are well known, see for example \cite{GMZ}. A grading on the Grassmann algebra is said to be a homogeneous grading if each element of $L$ is homogeneous. Let $A=\oplus_{g\in G}A_{g}$ be a $G$-graded algebra and $e$ the neutral element of $G$. The support of the grading is the set $Sup(A)=\{g\in G\mid A_{g}\neq 0\}$. If $Sup(A)=\{e\}$, we have the trivial grading on $A$.

We cite here the research developed by Anisimov \cite{anisimov1, anisimov2}, and by Di Vincenzo and Da Silva \cite{silcod, disil}. In \cite{claud-plamen-alan} the authors investigated the $\mathbb{Z}_{2}$ and $\mathbb{Z}$-graded central polynomials for $E$. In \cite{clau-laise-alan} the $\mathbb{Z}_{q}$-graded identities and central polynomials for all homogeneous $\mathbb{Z}_{q}$-gradings on $E$ were described, for $q>2$. In \cite{aagpk2} the authors developed the first study related to non-homogeneous $\mathbb{Z}_{2}$-gradings on $E$. They proved that there exist infinitely many non-homogeneous such gradings on $E$. On the other hand it turns out that in all ``typical'' cases the gradings obtained are very similar to homogeneous ones.  


In \cite{claud-plamen-alan, aagpk} the authors studied three types of  $\mathbb{Z}$-grading on $E$, which were denoted by $E^{\infty}$, $E^{k^\ast}$ and $E^{k}$, for all non-negative integer $k$. 
These structures are the more natural $\mathbb{Z}$-grading on $E$ due to its closely relation with the superalgebras $E_{\infty}$, $E_{k^\ast}$ and $E_{k}$. The support of $E^{k^\ast}$ is $\{0,1,\ldots, k\}$ while $E^{\infty}$ and $E^{k^\ast}$ have the set $\{0,1,\ldots\}$ as support. In this context, we formulated the questions: Is it possible to define a $\mathbb{Z}$-grading on $E$ whose support is all the group $\mathbb{Z}$? In affirmative case, is it possible to obtain some type of classification?

These questions were the main motivation to product this paper. The main results obtained here are:
\begin{enumerate}
	\item \textbf{Result 1}: We describe the $\mathbb{Z}$-graded polynomial identities for all $2$-induced $\mathbb{Z}$-gradings on $E$ of full support. We present a necessary and sufficient condition for two structures of this type to be $\mathbb{Z}$-isomorphic.
	
	\item \textbf{Result 2}: We provide examples of $\mathbb{Z}$-gradings on $E$ which are PI-equivalent but not $\mathbb{Z}$-isomorphic. This is the first example of PI-equivalence of graded algebra when the support of the grading is infinite.
	
	\item \textbf{Result 3}: We construct certain type of $\mathbb{Z}$-grading on $E$ whose $T_{\mathbb{Z}}$-ideal is closely related to $T_{2}$-ideal of $\mathbb{Z}_2$-gradings on $E$. 
\end{enumerate}

To simplify both the exposition and the notation, we consider the ground field $F$ of characteristic zero. Nevertheless  a lot of results obtained in Subsection  \ref{sub characterization} does not depend on the field $F$.

\section{Preliminaries}\label{Preliminares}
Throughout this paper all vector spaces will be considered over $F$,  and all algebras will be associative and unitary, and will be considered over $F$. We also fix an abelian group $G$. In this section we make a brief presentation of the tools used for the developing of this paper.

A $G$-grading on an algebra $A$ is a vector space decomposition $A=\oplus_{g\in G}A_{g}$ such that  $A_{g}A_{h}\subset A_{gh}$ for all $g$, $h\in G$. When a nonzero element $a\in A_{g}$ we say that $a$ is homogeneous and its degree is $\alpha(a)=g$ or $\|a\|=g$. A vector space (subalgebra, ideal) $B$ of $A$ is said to be a homogeneous vector space (subalgebra, ideal) if $B=\oplus_{g\in G}B\cap A_{g}$.
If $H$ is a subgroup of $G$ we define the quotient $G/H$-grading as $A=\oplus_{\overline{g}\in G/H}A_{\overline{g}}$ where $A_{\overline{g}}=\oplus_{h\in H}A_{gh}$. We say that two $G$-graded algebras $A=\oplus_{g\in G}A_{g}$ and $B=\oplus_{g\in G}B_{g}$ are $G$-isomorphic if there exists an isomorphism of algebras $\varphi : A\rightarrow B$ such that $\varphi(A_{g})=B_{g}$, for all $g\in G$. Now, let $A=\oplus_{g\in G}A_g$ and $B=\oplus_{h\in H}B_h$ be algebras graded by the groups $G$ and $H$ respectively. The graded algebras $A$ and $B$ are weakly isomorphic if there exists an isomorphism of groups $\rho:G\rightarrow H$ and an isomorphism of algebras $\varphi:A \rightarrow B$ such that $\varphi (A_g)=B_{\rho(g)}$ for every $g$ in $G$.

Group gradings on the Grassmann algebra with groups $G$ other than $\mathbb{Z}_2$ were studied in the papers \cite{centroneG, disilplamen}. In these the authors describe homogeneous gradings on $E$ by finite abelian groups and by cyclic groups of prime order, respectively. Along this line we will focus on $\mathbb{Z}$-gradings on $E$.

Let $ X=\cup_{i\in\mathbb{Z}}X_{i}$ be the disjoint union of infinite countable sets of variables $X_{i}=\{x_{1}^{i},x_{2}^{i},\ldots\}$, $i\in\mathbb{Z}$. Assuming that for each $i\in\mathbb{Z}$ the elements of the set $X_{i}$ are of $\mathbb{Z}$-degree $i$, the free associative algebra $F\langle X|\mathbb{Z}\rangle$ has a natural $\mathbb{Z}$-grading $\oplus_{i\in\mathbb{Z}} F\langle X|\mathbb{Z}\rangle^{i}$. Here $F\langle X|\mathbb{Z}\rangle^{i}$ is the vector subspace of $F\langle X|\mathbb{Z}\rangle$ spanned by all monomials of $\mathbb{Z}$-degree $i$. It is also easy to see that $F\langle X|\mathbb{Z}\rangle$ is free in the following sense. Given a $\mathbb{Z}$-graded algebra $A=\oplus_{i\in\mathbb{Z}} A_i$ and a map $g\colon X\to A$ such that $g(X_i)\subseteq A_i$ for every $i$ then $g$ can be uniquely extended to a homomorphism of algebras $\varphi\colon F\langle X|\mathbb{Z}\rangle\to A$ which respects the $\mathbb{Z}$-gradings on these algebras (such homomorphisms are called $\mathbb{Z}$-graded ones).

A polynomial $f(x_{1}^{l_1},\ldots, x_{r}^{l_r})$ is called a $\mathbb{Z}$-graded polynomial identity (PI) of the $\mathbb{Z}$-graded algebra $A$ if $f(a_{1},\ldots, a_{r})=0$ for all $a_{i}\in A_{l_i}$. The set $T_{\mathbb{Z}}(A)$ of all $\mathbb{Z}$-graded polynomial identities is an ideal of $F\langle X|\mathbb{Z}\rangle$. It is also invariant under $\mathbb{Z}$-graded endomorphisms of $F\langle X|\mathbb{Z}\rangle$. Such ideals are called $T_{\mathbb{Z}}$-ideals. As in the case of ordinary polynomial identities one proves that an ideal $I$ in $F\langle X|\mathbb{Z}\rangle$ is a $T_{\mathbb{Z}}$-ideal if and only if it coincides with the ideal of all $\mathbb{Z}$-graded identities of some $\mathbb{Z}$-graded algebra $A$. It is well known that studying ordinary polynomial identities in characteristic 0, one may consider the multilinear ones. The analogous fact holds for graded identities as well. More precisely if $A$ is a $\mathbb{Z}$-graded algebra over the field $F$ of characteristic 0, the ideal $T_{\mathbb{Z}}(A)$ of all $\mathbb{Z}$-graded identities of $A$ is generated as a $T_{\mathbb{Z}}$-ideal by its multilinear polynomials. In the more general case of an infinite field one has to take into account the multihomogeneous polynomials instead of the multilinear ones. 

If $A$ and $B$ are $\mathbb{Z}$-graded algebras we say that $A$ and $B$ are PI-equivalent as $\mathbb{Z}$-graded algebras if $T_{\mathbb{Z}}(A)=T_{\mathbb{Z}}(B)$. Given a $T_\mathbb{Z}$-ideal $I$ of $F\langle X|\mathbb{Z}\rangle$, the \textit{variety of $\mathbb{Z}$-graded algebras} $\mathfrak{V}^\mathbb{Z}$ associated to $I$ is the class of all the $\mathbb{Z}$-graded algebras $A$ such that $I$ is contained in $T_{\mathbb{Z}}(A)$. The $T_\mathbb{Z}$-ideal $I$ is denoted by $T_{\mathbb{Z}}(\mathfrak{V}^\mathbb{Z})$. The variety $\mathfrak{V}^\mathbb{Z}$ is generated by the $\mathbb{Z}$-graded algebra $A$ if $T_{\mathbb{Z}}(\mathfrak{V}^\mathbb{Z})=T_{\mathbb{Z}}(A)$. It is clear that the same concepts can be formulated for any abelian group $G$. Relavant works approached questions related to PI-equivalence of algebras, see for example \cite{divi pi, kemer pi}. 

Here we recall the more natural three types of $\mathbb{Z}$-gradings on $E$ that were presented in \cite{aagpk}; these are the natural analogues of the ones considered in \cite{disil}. To this end we consider the following attribution of degrees on its generators:
\begin{align*}
\|e_{i}\|^{k} & =\begin{cases} 0,\text{ if }  i=1,\ldots,k\\ 1, \text{ otherwise } 
\end{cases},\\	
\|e_{i}\|^{k^\ast} & =\begin{cases} 1,\text{ if }  i=1,\ldots,k\\ 0, \text{ otherwise } 
\end{cases},\\
\|e_{i}\|^{\infty} & =\begin{cases} 0,\text{ for }  i\text{ even }\\ 1, \text{ for i odd } 
\end{cases}.
\end{align*}	
Then we induce the $\mathbb{Z}$-grading on $E$ putting 
\[
\|e_{j_{1}}\cdots e_{j_{n}}\| =\|e_{j_{1}}\|+\cdots + \|e_{j_{n}}\|,
\]
and extend it to $E$ by linearity.

We denote by $E^{k}$, $E^{k^\ast}$ and $E^{\infty}$ the three types of $\mathbb{Z}$-graded algebras above. When $\|e_{i}\|=1$ for all $i$ (the $\mathbb{Z}$-grading $E^k$ for $k=0$), we denote the respective $\mathbb{Z}$-grading by $E^{can}$, which is called natural $\mathbb{Z}$-grading on $E$. 

The superalgebras $E_{k}$, $E_{k^\ast}$, and $E_{\infty}$ can be obtained respectively from $E^{k}$, $E^{k^\ast}$ and $E^{\infty}$ in a natural manner. 
We recall that the support of $E^{k^\ast}$ is $\{0, 1,\ldots, k\}$ while the supports of $E^{\infty}$ and $E^{k}$ are equal to $\{0,1,\ldots\}$.

In \cite{aagpk} the description of the graded polynomial identities for such gradings were presented, namely:

\begin{proposition}[see \cite{aagpk}]
	Let the field $F$ be of characteristic 0. 
	The $T_{\mathbb{Z}}$-ideal of the $\mathbb{Z}$-graded polynomial identities  of $E^{can}=\oplus_{n\in\mathbb{Z}} E^{(n)}$ is generated by the graded polynomials  
	\begin{itemize}
		\item $x$, if $\alpha(x)<0$.
		\item $[x_{1},x_{2}]$, if $\alpha(x_{1})$ or $\alpha(x_{2})$  is an even integer.
		\item $x_{1}x_{2} + x_{2}x_{1}$, if $\alpha(x_{1})$ and $\alpha(x_{2})$ are odd integers. 
	\end{itemize}
	 
\end{proposition}

\begin{theorem}[see \cite{aagpk}]
	Assume that $F$ is a field of characteristic 0. 
	Let $T_{\mathbb{Z}}(E^{d})$ be the $T_{\mathbb{Z}}$-ideal of the $\mathbb{Z}$-graded polynomial identities for $E^{d}$, where $d$ stands for either $\infty$ or $k^{\ast}$. Then  
	\begin{enumerate}
		\item $T_{\mathbb{Z}}(E^{\infty})$ is generated by the set of the following polynomials:
		\begin{itemize}
			\item $x$, if $\alpha(x)<0$.
			\item $[x_{1}, x_{2}, x_{3}]$, for every choice of the degrees $\alpha(x_{1})$, $\alpha(x_{2})$, $\alpha(x_{3})$. 
		\end{itemize}
		\item $T_{\mathbb{Z}}(E^{k^{\ast}})$ is generated by the set of the following polynomials:
		\begin{itemize} \label{identidades}
			\item $x$, if $\alpha(x)\notin\{0,\ldots, k\}$. 
			\item $[x_{1}, x_{2}, x_{3}]$, for every choice of the degrees $\alpha(x_{1})$, $\alpha(x_{2})$, $\alpha(x_{3})$.
		\end{itemize}
	\end{enumerate}
\end{theorem}

The description of the $\mathbb{Z}$-graded polynomial identities for $E^k$ is more complicated and can be found in  \cite{aagpk}.

Given a basic element (or monomial) $w=e_{i_1}\cdots e_{i_k}$ in $E$, we recall that the set $supp (w)=\{e_{i_1},\ldots , e_{i_k}\}$ is called \textit{support} of $w$. If $w_{1}=e_{j_1}\cdots e_{j_l}$ is in $E$, we say that $w$ and $w_{1}$ have \textit{pairwise disjoint supports} if $supp (w)\cap supp (w_{1})=\emptyset$. In this case we have that $ww_{1}\neq 0$. The length of $w$ is the cardinality of $supp (w)$, and we denote it by $|w|$. In the grading $E^{can}$, the $\mathbb{Z}$-degree of $w$ is exactly $|w|$. 

\begin{definition}
	A $G$-grading $E=\oplus_{g\in G}E_{g}$ on the Grassmann algebra is said to be:
	\begin{enumerate}
		\item [(a)] A homogeneous grading if each $e_{i}\in L$ is homogeneous in the grading.
		\item [(b)] A grading of full support if its support is all the group $G$.
	\end{enumerate}

\end{definition}

The Grassmann algebra $E$ admits infinitely many non-homogeneous $\mathbb{Z}_{2}$-gradings, as can be seen in \cite{aagpk2}. It is clear that $E^{k}$, $E^{k^\ast}$, and $E^{\infty}$ are homogeneous gradings on $E$, but these structures are not of full support. 

Throughout the rest of the paper, unless otherwise stated, all the gradings on $E$ will be considered homogeneous gradings.

We write $\mathbb{N}=\{1, 2, 3,\ldots\}$ and $\mathbb{N}_{0}=\mathbb{N}\cup \{0\}$.
Let $r\in\mathbb{N}$, $n_{1}<\ldots <n_{r}$, with each $n_{j}\in\mathbb{Z}$, and $v_{1},\ldots, v_{r}$, where $v_{j}\in\mathbb{N}$ or $v_{j}=\infty$, for $1\leq j\leq r$. Consider 
\[L=L_{n_{1}}^{v_{1}}\oplus\cdots\oplus L_{n_{r}}^{v_{r}}\]
a decomposition of $L$ in $r$ subspaces such that $v_{j}=\dim L_{n_{j}}^{v_{j}}$. Clearly we can assume that the generators $e_i$ of $L$ satisfy $e_i\in \cup_{j=1}^r L_{n_j}^{v_j}$. In other words we split the basis $e_1$, $e_2$, \dots{} of $L$ into $r$ disjoint sets. Given a generator $e_{k}\in L$, we define 
\[
\|e_{k}\|=n_{j}\quad \mbox{\textrm{ if and only if }}\quad  e_{k}\in L_{n_j}^{v_j}.
\]

We extend the degree for all monomial $e_{k_1}e_{k_2}\cdots e_{k_s}$ by
\[
\|e_{k_1}e_{k_2}\cdots e_{k_s}\|=\|e_{k_1}\|+\|e_{k_2}\|+\cdots +\|e_{k_s}\|.
\]
It provides us a structure of $\mathbb{Z}$-grading on $E$, denoted by $E_{(n_{1},\ldots, n_{r})}^{(v_{1},\ldots, v_{r})}$. In this case we say that $E_{(n_{1},\ldots, n_{r})}^{(v_{1},\ldots, v_{r})}$ is a \textit{$r$-induced} $\mathbb{Z}$-grading. We say that $n_{1},\ldots, n_{r}$ are the \textit{lower indexes} and $v_{1},\ldots, v_{r}$ are the \textit{upper indexes} of the grading.

\begin{remark}\label{induzindo em qualquer G}
	The procedure just described is more general. Let $G$ be a group, consider $g_{1},\ldots, g_{r}$ in $G$ and suppose that $g_{i}g_{j}=g_{j}g_{i}$, for all $i, j\in \{1, \ldots, r\}$. As above we consider 
	\[L=L_{g_{1}}^{v_{1}}\oplus\cdots\oplus L_{g_{r}}^{v_{r}}\]
	a decomposition of $L$ in $r$ subspaces such that $v_{j}=\dim L_{g_{j}}^{v_{j}}$. Then we can define a $G$-grading $E_{(g_{1},\ldots, g_{r})}^{(v_{1},\ldots, v_{r})}$ on $E$ exactly as before. Observe that the condition $g_{i}g_{j}=g_{j}g_{i}$, for all $i, j\in \{1, \ldots, r\}$ is essential.
\end{remark}

\begin{example}[The trivial case]
	If $r=1$, the respective $\mathbb{Z}$-graded polynomial identities for $E_{(n)}^{(\infty)}$ were described in \cite{aagpk}. It is easy to see that $E_{(n)}^{(\infty)}$ and $E_{(m)}^{(\infty)}$ are $\mathbb{Z}$-isomorphic if and only if $n=m$. 
\end{example}

\section{$2$-induced $\mathbb{Z}$-gradings  of full support}\label{start}

From now on we consider $r=2$ and the respective decomposition of $L$ has the form 
\[L=L_{m}^{u}\oplus L_{n}^{v}.\]

\subsection{A characterization}\label{sub characterization}

In this subsection we shall give conditions on the indexes $m$, $n$, $u$ and $v$ so that $Sup(E_{(m, n)}^{(u, v)})=\mathbb{Z}$. After that we shall describe the $\mathbb{Z}$-graded polynomial identities for all these structures. We call attention that the construction exposed in this subsection does not depend on the field $F$.

We start with a result that characterizes $Sup(E_{(m, n)}^{(u, v)})$ when $u=v=\infty$.

\begin{theorem}\label{corbom} 
	Let $L=L_{m}^{u}\oplus L_{n}^{v}$ be a decomposition as previously presented. Then $Sup(E_{(m, n)}^{(u, v)})=d\mathbb{Z}$ if and only if $u=v=\infty$, $m<0<n$ and $\gcd(m, n)=d.$
\end{theorem}
\begin{proof}
	Assume that $Sup(E_{(m, n)}^{(u, v)})=d\mathbb{Z}$ and $m<n$. 
	If $m, n\geq 0$, then it is clear that $Sup(E_{(m, n)}^{(u, v)})\subset \mathbb{N}_{0}$. If $m, n\leq 0$ we conclude that $Sup(E_{(m, n)}^{(u, v)})\subset \{-k\mid k\in\mathbb{N}_{0}\}$. 
	Hence it follows that $m<0<n$.
	If $u=\dim L_{m}^{u}<\infty$, we conclude that $Sup(E_{(m, n)}^{(u, v)})\subset \{a\in\mathbb{Z}\mid a\geq mu\}$ and a similar conclusion holds if $v=\dim L_{n}^{v}<\infty$. Moreover, as $m, n\in Sup(E_{(m, n)}^{(u, v)})=d\mathbb{Z}$, we have that $d^\prime=\gcd(m,n)\in d\mathbb{Z}$, since $d$ divide $m$ and $n$. Hence $d^\prime\mathbb{Z}\subseteq d\mathbb{Z}$. On the other hand $d\in Sup(E_{(m, n)}^{(u, v)})$ then there exist $r$ and $s$ natural numbers such that $d=mr+ns$ this implies that $d\in d^\prime\mathbb{Z}$. Hence $d=\gcd(m,n)$.
	
	Now if we consider $\gcd(m, n)=d$, then $\gcd(m/d, n/d)=1$. We have that $\overline{(m/d)}$ is a generator of the group $\mathbb{Z}_{n/d}$. Since $\mathbb{Z}_{n/d}$ is finite, there exist positive integers $\alpha$, $\alpha '$ such that $\alpha(\overline{m/d})=\overline{1}$ and $\alpha'(\overline{m/d})=-\overline{1}$ in $\mathbb{Z}_{n/d}$. Therefore, there exist integers $\beta$, $\beta'$ such that 
	\[1-\alpha (m/d) =\beta (n/d),\]
	\[-1-\alpha' (m/d) =\beta' (n/d).\]
	As $m<0$ we have $1-\alpha (m/d)\geq 0$ and $-1-\alpha' (m/d)\geq 0$, it follows that $\beta, \beta'\in\mathbb{N}_0$. Hence 
	\[d=\alpha m + \beta n,\]
	\[-d=\alpha' m + \beta' n.\]
	This and $u=v=\infty$ imply that $d\mathbb{Z}\subseteq Sup(E_{(m, n)}^{(u, v)})$. On the other hand, if $k\in Sup(E_{(m, n)}^{(u, v)})$, there exist $r$, $s\in \mathbb{N}_0$ such that $k=rm+sn$. It follows from definition of greatest common divisor that $d$ divide $k$, and hence we have the equality $d\mathbb{Z}=Sup(E_{(m, n)}^{(u, v)})$.
\end{proof}

The next result is a consequence of Theorem \ref{corbom} and it gives us a necessary and sufficient condition for the $\mathbb{Z}$-graded algebra $E_{(m, n)}^{(u, v)}$ to be of full support. 

\begin{corollary}\label{criterio full sup}
	Let $L=L_{m}^{u}\oplus L_{n}^{v}$ be a decomposition as previously presented. The following statements are equivalent:
	\begin{enumerate}
		\item $Sup(E_{(m, n)}^{(u, v)})=\mathbb{Z}$.
		\item $u=v=\infty$, $m<0<n$ and there exist $\alpha$, $\beta$, $\alpha '$ and $\beta '$ in $\mathbb{N}_{0}$ satisfying $1=\alpha m + \beta n$ and $-1=\alpha ' m + \beta ' n$.
		\item $u=v=\infty$, $m<0<n$ and $\gcd(m, n)=1.$
	\end{enumerate}
\end{corollary}
\begin{proof}
It is enough to prove that ($2 \Rightarrow 1$). Let $k>0$ be an integer number and let $A_{k}$ be the respective component in the grading. Since $1=\alpha m + \beta n$ and $u=v=\infty$, we can take $k$ elements  $w_1,\ldots, w_k$ of $B_E$ with pairwise disjoint supports of degree $1$. Then $0\neq w_1\cdots w_k$ has degree $k$, that is, $ A_{k}\neq 0$. The case $k<0$ is treated similarly. 
\end{proof}

\begin{example}\label{ex -2, 3}
	Due to Corollary \ref{criterio full sup}, we can guarantee that $E_{(-100, 147)}^{(\infty, \infty)}$ and $E_{(-2, 3)}^{(\infty, \infty)}$ are of full support. Notice that
	\[1=1\times(-2) + 1\times 3,\]
	\[-1=2\times(-2)+1\times 3.\]
	If $A_{r}$ is the homogeneous component of degree $r$ in $E_{(-2, 3)}^{(\infty, \infty)}$, we observe that $A_1$ has infinitely many monomials of length 2 and $A_{-1}$ has infinitely many monomials of length 3. 
\end{example}

\begin{corollary}\label{muitos de mesmo length}
	Suppose that $E_{(m, n)}^{(u, v)}$ is of full support and let $A_{r}$ be the homogeneous component of degree $r$. Given a monomial $w\in A_{r}$, there exist infinitely many monomials with pairwise disjoint supports in $A_{r}$ whose length is equal to the length of $w$.
\end{corollary}
\begin{proof}
	Assume that  
	\[L_{m}^{ \infty}=span_{F}\{e_{1}, e_{3}, e_{5},\ldots\},\] 
	\[L_{n}^{ \infty}=span_{F}\{e_{2}, e_{4}, e_{6},\ldots\},\]
		\[w=e_{i_1}\cdots e_{i_l}e_{j_1}\cdots e_{j_k},\]
	where $i_{1},\ldots, i_{l}$ are odd numbers and $j_{1},\ldots, j_{k}$ are even numbers.
	
	In this case for any choice of $l$ odd integers and $k$ even integers we can construct a monomial in $A_{r}$ with the same length of $w$.
\end{proof}


	
\begin{remark}\label{weakisoid}
	Let $G$ and $H$ be groups and let $A=\oplus_{g\in G}A_{g}$ and  $B=\oplus_{h\in H}B_h$ be algebras graded by the groups $G$ and $H$, respectively. If $A$ and $B$ are weakly isomorphic with respect to the isomorphism of groups $\rho\colon G\to H$ then a polynomial $f\in F\langle X|G \rangle$ is a graded identity for $A$ if and only if $\Phi (f)$ is a graded identity for $B$, where $\Phi: F\langle X|G \rangle \rightarrow F\langle X|H \rangle$ is the isomorphism such that $\Phi (x_{i}^{g})=x_{i}^{\rho(g)}$. Moreover if $S$ is a basis for the $T_G$-ideal $T_G(A)$ then $\Phi(S)$ is a basis for the $T_H$-ideal $T_H(B)$.
\end{remark}

Given $d\in \mathbb{Z}$, we consider the map $\Phi_d\colon F\langle X|\mathbb{Z}\rangle\to F\langle X|d\mathbb{Z}\rangle$ defined by $\Phi_d(x_i^n)=x_i^{dn}$. If $d\neq 0$, it is clear that the map $\rho\colon \mathbb{Z}\to d\mathbb{Z}$ given by $\rho(n)=dn$ is an isomorphism of groups. Due to these comments and Remark \ref{weakisoid}, we obtain the following lemma.
\begin{lema}\label{good}
	Let $f\in F\langle X|\mathbb{Z}\rangle$ be a multilinear polynomial and $m<0<n$.  Then $f\in T_{\mathbb{Z}}(E_{(\frac{m}{d}, \frac{n}{d})}^{(\infty, \infty)})$ if and only if $\varPhi_d(f)\in T_{d\mathbb{Z}}(E_{(m, n)}^{(\infty, \infty)})$, where $d=\gcd(m,n)$.
\end{lema}

Hence if we consider $E_{(m, n)}^{(\infty, \infty)}$ as a $d\mathbb{Z}$-graded algebra, its $d\mathbb{Z}$-graded identities can be obtained from the identities of the $\mathbb{Z}$-graded algebra $E_{(\frac{m}{d}, \frac{n}{d})}^{(\infty, \infty)}$.  


We shall present a relation between $T_{\mathbb{Z}}(E_{(m, n)}^{(\infty, \infty)})$ and $T_{\mathbb{Z}}(E_{(\frac{m}{d}, \frac{n}{d})}^{(\infty, \infty)})$, where $d=\gcd(m,n)$. To this end we define the homomorphism $\Psi_d \colon F\langle X|d\mathbb{Z}\rangle\to F\langle X|\mathbb{Z}\rangle$ given by $\Psi_d(x_i^{dn})=x_i^{dn}$. Note that this homomorphism is not inclusion, because even though they appear to be the same variables they are in distinct free algebras. 

\begin{theorem}\label{idquasefull}
	Let $m$ and $n$ be integer numbers, with $m<0<n$ and $d=\gcd(m,n)$. If $S$ is a basis for the $T_\mathbb{Z}$-ideal $T_{\mathbb{Z}}(E_{\left( \frac{m}{d},\frac{n}{d}\right) }^{(\infty, \infty)})$, then the $T_\mathbb{Z}$-ideal $T_{\mathbb{Z}}(E_{(m, n)}^{(\infty, \infty)})$ is generated by the set $S^\prime\cup N$, where 
	$$S^\prime=\{\Psi_d(\varPhi_d(f))\mid f\in S\}$$
	and
	$$N=\{x\in F\langle X|\mathbb{Z}\rangle\mid \alpha(x)\notin d\mathbb{Z}\}.$$
\end{theorem} 
\begin{proof}
	Let $f$ be an element in $S^\prime$. There exists $g\in S$ such that $\Psi_d(\varPhi_d(g))=f$. By the lemma mentioned above, we have $\varPhi_d(g)\in T_{d\mathbb{Z}}(E_{(m, n)}^{(\infty, \infty)})$. Notice that the elements in $Sup(E_{(m, n)}^{(\infty, \infty)})$ in $d\mathbb{Z}$ can be seen in $\mathbb{Z}$, and hence we have that the elements of degree $i$ in $Sup(E_{(m, n)}^{(\infty, \infty)})$ (in $d\mathbb{Z}$) continue to have degree $i$ in $\mathbb{Z}$. We conclude that $\Psi_d(\varPhi_d(g))\in T_{\mathbb{Z}}(E_{(m, n)}^{(\infty, \infty)})$. Moreover it is clear that $N \subseteq T_{\mathbb{Z}}(E_{(m, n)}^{(\infty, \infty)})$. Hence $\langle S^\prime\cup N\rangle_{T_{\mathbb{Z}}}\subseteq T_{\mathbb{Z}}(E_{(m, n)}^{(\infty, \infty)})$.
	
	Now we consider $f=f(x_1^{l_1},\ldots, x_r^{l_r})\in T_{\mathbb{Z}}(E_{(m, n)}^{(\infty, \infty)})$ multilinear. If $l_i\in \mathbb{Z}\setminus Sup(E_{(m, n)}^{(\infty, \infty)})$, for some $i=1,\ldots, r$, then $f\in \langle S^\prime\cup N\rangle_{T_{\mathbb{Z}}}$. Hence we can suppose $l_i\in Sup(E_{(m, n)}^{(\infty, \infty)})$,  for every $i=1$, \ldots, $r$. In this case, for each $i$, there exists $s_i\in \mathbb{Z}$ such that $l_i=ds_i$. 
	In this context, we can consider, without loss of generality, that $f\in F\langle X|d\mathbb{Z}\rangle\cap T_{\mathbb{Z}}(E_{(m, n)}^{(\infty, \infty)})=T_{d\mathbb{Z}}(E_{(m, n)}^{(\infty, \infty)})$. Lemma \ref{good} implies $f=\varPhi_d(g)$, where $g=f(x_1^{s_1},\ldots, x_r^{s_r})$
	and $g\in T_{\mathbb{Z}}(E_{(\frac{m}{d}, \frac{n}{d})}^{(\infty, \infty)})$. As $T_{\mathbb{Z}}(E_{(\frac{m}{d}, \frac{n}{d})}^{(\infty, \infty)})$ is generated by $S$, then there exist $g^j_1$, $g^j_2$, $h^j_1$, \ldots, $h^j_{p_j} \in F\langle X|\mathbb{Z}\rangle$ and $t_1$, \ldots, $t_u\in S$ such that
	$$g=\sum_j g^j_1t_j(h^j_{1},\ldots,h^j_{p_j}) g^j_{2}.$$
	In this case
	$$f=\Psi_d(\varPhi_d(g))=\sum_j G^j_1[\Psi_d(\varPhi_d(t_j))(H^j_{1},\ldots,H^j_{p_j})] G^j_{2}.$$
	where $G^j_1, G^j_2, H^j_{1}, \ldots, H^j_{p_j}\in F\langle X|\mathbb{Z}\rangle$ with $G^j_i=\Psi_d(\varPhi_d(g^j_i))$ and $H^j_k=\Psi_d(\varPhi_d(h^j_k))$, for $i=1,2$ and $k=1,\ldots, p_j$. Therefore we have that $f\in \langle S^\prime\cup N\rangle_{T_{\mathbb{Z}}}$, and we are done. 
\end{proof}

The last theorem implies that to determine the $\mathbb{Z}$-graded identities for an algebra of type $E_{(m, n)}^{(\infty, \infty)}$, when $m<0<n$, it is sufficient to consider the case in which such structure is of full support, namely the case where $\gcd(m, n)=1$. 

 Suppose that $Sup(E_{(m, n)}^{(\infty, \infty)})=\mathbb{Z}$. By Corollary \ref{criterio full sup} there exist $\alpha$, $\beta$, $\alpha '$ and $\beta '$ in $\mathbb{N}_{0}$ satisfying 
 \[1=\alpha m + \beta n\]
 \[-1=\alpha ' m + \beta ' n.\]
 
 We fix $\alpha$, $\beta$, $\alpha'$  and $\beta'$ satisfying these equalities such that the sums $\alpha +\beta$ and $\alpha' +\beta'$ are minimal. We say that $\alpha +\beta$ and $\alpha' +\beta'$ are the \textit{coefficients} of the grading. We have to take into account the following four types:
 \begin{enumerate}
 	\item [(1)] Both $\alpha + \beta$ and $\alpha' + \beta'$ are even.
 	\item [(2)] $\alpha + \beta$ is even and $\alpha' + \beta'$ is odd.
 	\item [(3)] $\alpha + \beta$ is odd and $\alpha' + \beta'$ is even.
 	\item [(4)] Both $\alpha + \beta$ and $\alpha' + \beta'$ are odd.
 \end{enumerate}
For instance, the grading of $E_{(-2, 3)}^{(\infty, \infty)}$ given in Example \ref{ex -2, 3} is of type $(2)$. 
\begin{remark}\label{the pi-equivalence}
	Let $m<0<n$ be integers. If both $m$ and $n$ are odd then $E_{(-1, n)}^{(\infty, \infty)}$ and $E_{(m, 1)}^{(\infty, \infty)}$ are of type $(4)$. In fact, for $E_{(-1, n)}^{(\infty, \infty)}$ we have that
	\[\alpha + \beta=n,\]
	\[\alpha' + \beta'=1.\]
	And similarly for $E_{(m, 1)}^{(\infty, \infty)}$. 
	If $n$ is even then $E_{(-1, n)}^{(\infty, \infty)}$ is of type $(2)$. If $m$ is even then $E_{(m, 1)}^{(\infty, \infty)}$ is of type $(3)$. 
\end{remark}

  \subsection{Type (1)}
  
  Here and in the rest of the paper, we shall consider the ground field $F$ of characteristic zero. It is possible to obtain similar results when $F$ is infinite of characteristic $p>2$, but the arguments and notations become rather clumsy. 
  
  Let $E_{(m, n)}^{(\infty, \infty)}$ be a structure of type (1).
  Let  $\mathfrak{V}_{1}^\mathbb{Z}$ be the variety  of $\mathbb{Z}$-graded algebras define by
  \begin{center}
  	\begin{enumerate} 		
  		\item $[x_{1}, x_{2}, x_{3}]$, for all degree of $x_{1}$, $x_{2}$, $x_{3}$. 
  	\end{enumerate} 
  \end{center}
  
  Obviously we have that $E_{(m, n)}^{(\infty, \infty)}\in \mathfrak{V}_{1}^\mathbb{Z}$. Besides $E^{\infty}\in \mathfrak{V}_{1}^\mathbb{Z}$ but it does not generate $\mathfrak{V}_{1}^\mathbb{Z}$.
  
  \begin{proposition}\label{V1}
  	Let $E_{(m, n)}^{(\infty, \infty)}$ be a structure of type $(1)$. Then $E_{(m, n)}^{(\infty, \infty)}$ generates the variety $\mathfrak{V}_{1}^\mathbb{Z}$. 
  \end{proposition}
  
  To prove Proposition \ref{V1} it is necessary to understand what happens in each homogeneous component of $E_{(m, n)}^{(\infty, \infty)}=\bigoplus_{r\in\mathbb{Z}}A_{r}$.
  
  \begin{lema}\label{m, n odd implies type 4}
  	If $m$ and $n$ are odd integers then $E_{(m, n)}^{(\infty, \infty)}$ is of type $(4)$.
  \end{lema}
  \begin{proof}
  	In fact, the equalities
  	\[1=\alpha m + \beta n,\]
  	\[-1=\alpha ' m + \beta ' n\]
  	imply that $\alpha$ and $\beta$ have contrary parities (and also $\alpha'$ and $\beta'$). Then both $\alpha +\beta$ and $\alpha' +\beta'$ are odd, hence $E_{(m, n)}^{(\infty, \infty)}$ is of type $(4)$.
  \end{proof}

  If $E_{(m, n)}^{(\infty, \infty)}=\bigoplus_{r\in\mathbb{Z}}A_{r}$ is of type $(1)$, due to Lemma \ref{m, n odd implies type 4}, we have that $m\times n$ is even. Therefore, as $\gcd(m, n)=1$, it follows that $-m+n$ is odd.
  
  \begin{lema}\label{A0 inf even and odd1}
  	The component $A_{0}$ has infinitely many elements of $B_E$ with even and odd length with pairwise disjoint supports. Consequently the same property is valid for each component $A_{r}$, for all $r\in\mathbb{Z}$.
  \end{lema}
  \begin{proof}
  	As $\alpha +\beta$ is even, there exist infinitely many monomials of even length with pairwise disjoint supports in $A_{1}$. Also this is true for $A_{-1}$. Then $A_{0}$ has infinitely many monomials of even length with pairwise disjoint supports. 
  	
  	On the other hand, we have that
  	\[0=n\times m + (-m)\times n,\]
  	therefore $A_{0}$ has infinitely many monomials of length $n+(-m)$ (odd number).
  	
  	Now given a monomial $w\in A_{r}$, due to Lemma \ref{muitos de mesmo length}, we have that $A_{r}$ has infinitely many monomials with pairwise disjoint supports and length equal to the length of $w$. Since $wA_{0}\subset A_r$, the last statement follows. 
  \end{proof}

  Given a positive integer $t$ and $n_{1}$, \dots, $n_{t}$ in $\mathbb{Z}$, we consider the positive integers $l_{n_1}$, \dots, $l_{n_t}$ and $l=l_{n_1}+ \cdots +l_{n_t}$. Let us denote by $V_{l_{n_1},\ldots, l_{n_t}}$ the vector space of $\mathbb{Z}$-graded multilinear polynomials having: 
  
  \begin{tabular}{rll}
  	$l_{n_1}$ & variables of degree  $n_1$, & namely $x_{1}^{n_1}$, \dots, $x_{l_{n_1}}^{n_1}$,\\
  	$\vdots$ &&\\ 
  	$l_{n_t}$ &  variables of degree $n_t$, & namely $x_{1}^{n_t}$, \dots, $x_{l_{n_t}}^{n_t}$.
  \end{tabular}
  
  Define $\psi\colon  V_{l}\to V_{l_{n_1},\ldots, l_{n_t}} $ as the linear isomorphism induced by the map 
  \[
  x_{i}\mapsto   
  \begin{cases} x_{i}^{n_1},\text{ if }  1\leq i\leq l_{n_1} \\  
  x_{i-l_{n_1}}^{n_2},\text{ if } l_{n_1}+1\leq i\leq l_{n_1}+l_{n_2} \\ 
  \vdots \\ 
  x_{i-(l_{n_1}+\cdots + l_{n_{t-1}})}^{n_t},\text{ if } l_{n_1}+\cdots + l_{n_{t-1}}+1\leq i\leq l_{n_1}+\cdots + l_{n_{t-1}}+l_{n_t}  
  \end{cases}.
  \]
  
  
  \begin{remark}\label{correspon idem}
  	Consider $E_{(m, n)}^{(\infty, \infty)}$ of type $(1)$. Since the field $F$ has characteristic zero, following the same ideas that were used in \cite[Proposition 5]{aagpk}, it is possible to show that
  	\[
  	\psi(V_{l}\cap T(E))=V_{l_{n_1},\ldots, l_{n_t}}\cap T_{\mathbb{Z}}(E_{(m, n)}^{(\infty, \infty)}),
  	\]
  	where $T(E)$ is the ordinary $T$-ideal of $E$. Such correspondence comes from the fact that each homogeneous component of $E_{(m, n)}^{(\infty, \infty)}$ admits infinitely many monomials of even and odd length with pairwise disjoint supports, see Lemma \ref{A0 inf even and odd1}. Hence a $\mathbb{Z}$-graded substitution is essentially an ordinary substitution. In general, if a $\mathbb{Z}$-grading on $E$ has infinitely many monomials of even and odd length with pairwise disjoint supports in each homogeneous component then the above correspondence is still valid.
  	
  \end{remark}
  
  Now we are able to prove Proposition \ref{V1}:
  
  \begin{proof}
  	It is clear that $E_{(m, n)}^{(\infty, \infty)}\in\mathfrak{V}_{1}^\mathbb{Z}$.
  	
  	Given a multilinear $\mathbb{Z}$-graded polynomial identity $f(x_{1}^{n_1},\ldots,x_{l_{n_t}}^{n_t})$ of $E_{(m, n)}^{(\infty, \infty)}$, by Remark \ref{correspon idem} we have that $f(x_{1},\ldots, x_{l})$ is an ordinary polynomial identity of $E$. Hence $f=\sum_{i=1}^{h}a_{i}[b_{i}, c_{i}, d_{i}]g_{i}$ for some polynomials $a_{i}$, $b_{i}$, $c_{i}$, $d_{i}$, $g_{i}\in F\langle X\rangle$. As $f (x_1,\ldots, x_l)$ is multilinear we can assume that each of these elements is a monomial in $F\langle X\rangle$,  and $a_{i}b_{i}c_{i}d_{i}g_{i}\in V_{l}$ for every $i=1$, \dots, $h$. Therefore we obtain that  
  	\[f(x_{1}^{n_1},\ldots,x_{l_{n_t}}^{n_t})=\psi(f (x_1,\ldots, x_l))=\sum_{i=1}^{h}\psi(a_{i}[b_{i}, c_{i}, d_{i}]g_{i})=\]
  	\[=\sum_{i=1}^{h}\overline{a_{i}}[\overline{b_{i}}, \overline{c_{i}}, \overline{d_{i}}]\overline{g_{i}}.\]  
  	Here $\overline{a_{i}}$, $\overline{b_{i}}$, $\overline{c_{i}}$, $\overline{d_{i}}$, $\overline{g_{i}}$ are monomials in $F\langle X|\mathbb{Z}\rangle$ and $\overline{a_{i}}\overline{b_{i}} \overline{c_{i}}\overline{d_{i}}\overline{g_{i}}\in V_{l_{n_1},\ldots, l_{n_t}}$, and we are done.		
  \end{proof}

\subsection{Types (2) and (3)}
Let $E_{(m, n)}^{(\infty, \infty)}$ be a structure of type either (2) or (3).

\begin{lema}\label{A0 inf even and odd}
	The component $A_{0}$ has infinitely many elements of $B_E$ with even and odd length with pairwise disjoint supports. Consequently the same property is valid for each component $A_{r}$, for all $r\in\mathbb{Z}$.
\end{lema}
\begin{proof}
	Suppose that $E_{(m, n)}^{(\infty, \infty)}$ is of type $(2)$. Then $A_{1}$ has infinitely many elements of $B_E$ with pairwise disjoint supports and even length (length $\alpha + \beta$) and $A_{-1}$ has elements of $B_E$ of odd length (length $\alpha' + \beta'$) in the same conditions. Therefore $A_0$ has infinitely many elements of $B_E$ with pairwise disjoint supports and both length (because $A_{-1}A_{1}\subset A_{0}$ and $A_{-1}A_{1}A_{-1}A_{1}\subset A_{0}$).
	
	Now given a monomial $w\in A_{r}$, due to Corollary \ref{muitos de mesmo length}, we have that $A_{r}$ has infinitely many monomials with pairwise disjoint supports and length equal to the length of $w$. Since $wA_{0}\subset A_r$, the last statement follows. When $E_{(m, n)}^{(\infty, \infty)}$ is of type $(3)$ we use the same argument. 
\end{proof}

\begin{proposition}\label{V11}
	Let $E_{(m, n)}^{(\infty, \infty)}$ be a structure of type $(2)$ or $(3)$. Then $E_{(m, n)}^{(\infty, \infty)}$ generates the variety $\mathfrak{V}_{1}^\mathbb{Z}$. 
	 
\end{proposition}
\begin{proof}
	It follows from Remark \ref{correspon idem} and Lemma \ref{A0 inf even and odd}. 
\end{proof}

 \subsection{Type (4)}

This section  is devoted to study of the structures $E_{(m, n)}^{(\infty, \infty)}$ of type (4). 

Suppose that $E_{(m, n)}^{(\infty, \infty)}$ comes from the decomposition 
$L=L_{m}^{\infty}\oplus L_{n}^{\infty}$. 
We may assume that 
\[L_{m}^{ \infty}=span_{F}\{e_{1}, e_{3}, e_{5},\ldots\},\] 
\[L_{n}^{ \infty}=span_{F}\{e_{2}, e_{4}, e_{6},\ldots\}.\]

Consider the variety $\mathfrak{V}_{2}^\mathbb{Z}$ of $\mathbb{Z}$-graded algebras define by

\begin{center}
	\begin{enumerate} 
		\item	$[x_{1}, x_{2}]$, if either $\alpha(x_{1})$ or $\alpha(x_{2})$ is even and
		\item $x_{1}x_{2}+x_{2}x_{1}$, if both $\alpha(x_{1})$ and $\alpha(x_{2})$ are odd.
	\end{enumerate} 
\end{center} 

\begin{example}
	Note that $E^{can}\in \mathfrak{V}_{2}^\mathbb{Z}$ but it does not generate $\mathfrak{V}_{2}^\mathbb{Z}$.
\end{example}

In the following results an essential condition for our computation is assumed: the indexes $m$ and $n$ will be considered odd integers. 

\begin{lema}\label{A in centro}
	Suppose that $m$ and $n$ are odd integers, let $E_{(m, n)}^{(\infty, \infty)}=\bigoplus_{r\in\mathbb{Z}}A_{r}$ and $E_{can}=E_{(0)}\oplus E_{(1)}$. If $r$ is an even integer then $A_{r}\subset Z(E)$. If $r$ is an odd integer then $A_{r}\subset E_{(1)}$.
\end{lema}
\begin{proof}
	Let $r\geq 0$ be an integer. Let
	\[w=e_{i_1}\cdots e_{i_l}e_{j_1}\cdots e_{j_k}\]
	be a monomial of $E$. Clearly we can assume that $i_{1},\ldots, i_{l}$ are odd numbers and $j_{1},\ldots, j_{k}$ are even numbers. Note that
	\[w\in A_{r}\quad \mbox{\textrm{ if and only if }}\quad r=lm+kn.\]
	If $r$ is even, since $m$ and $n$ are odd, we conclude that $l$ and $k$ have the same parity. Thus $l+k$ is even and $A_{r}\subset Z(E)$. Similarly we show that $A_{-r}\subset Z(E)$.
	
	If $r$ is odd, we conclude that $l$ and $k$ have contrary parities, i.e, $A_{r}\subset E_{(1)}$. The same idea implies that $A_{-r}\subset E_{(1)}$ as well.
\end{proof}

\begin{theorem}
	Let $m$ and $n$ be odd integers. Then $E_{(m, n)}^{(\infty, \infty)}$ generates $\mathfrak{V}_{2}^\mathbb{Z}$.
\end{theorem}
\begin{proof}
	It is a trivial consequence of Lemma \ref{A in centro}.
\end{proof}

Now suppose that $m\times n$ is even. Hence we have that $-m+n$ is an odd integer. Since 
\[0=n\times m + (-m)\times n,\]
the component $A_{0}$ has infinitely many monomials of odd length, namely of length $n+(-m)$.

As $1=\alpha m +\beta n $ and $-1=\alpha' m +\beta' n $, the components $A_1$ and $A_{-1}$ have monomials in the same conditions of odd length. 

Then $A_{0}$ has infinitely many monomials of both length with pairwise disjoint supports. Therefore this is also valid for each homogeneous component of the grading. In the light of this fact, we have that:

\begin{theorem}
	If $m\times n$ is even and $E_{(m, n)}^{(\infty, \infty)}$ is of type $(4)$, then $E_{(m, n)}^{(\infty, \infty)}$ generates the variety $\mathfrak{V}_{1}^\mathbb{Z}$.
\end{theorem}
\begin{proof}
	Once again, it is enough to apply Remark \ref{correspon idem} and the previous comments. 
\end{proof}

In the next box, the results of this section are summarized:

\[
\begin{tabular}{|c|c|c|c|c|}
\hline
$\gcd(m,n)=1$ & $m\times n$ even & $E_{(m, n)}^{(\infty, \infty)}$ generates $\mathfrak{V}_{1}^\mathbb{Z}$\\\hline
$\gcd(m,n)=1$ & $m\times n$ odd & $E_{(m, n)}^{(\infty, \infty)}$ generates $\mathfrak{V}_{2}^\mathbb{Z}$ \\\hline
\end{tabular}
\]


 \section{$\mathbb{Z}$-isomorphism and PI-equivalence}
     In this section, all the $2$-induced $\mathbb{Z}$-gradings on $E$ of full support shall be classified. First observe that $E_{(-1, 1)}^{(\infty, \infty)}$ generates the variety $\mathfrak{V}_{2}^{\mathbb{Z}}$ and $E_{(-1, 2)}^{(\infty, \infty)}$ generates the variety $\mathfrak{V}_{1}^{\mathbb{Z}}$.

     In the light of this and also of the Section \ref{start} it can be enunciated that:
     
      \begin{theorem}\label{pi equiv here}
     Consider the Grassmann algebra $E$ over a field $F$ of characteristic zero. Let $m<0<n$ be integer numbers and suppose that $\gcd(m, n)=1$. 
     \begin{enumerate}
     \item The $\mathbb{Z}$-grading $E_{(m, n)}^{(\infty, \infty)}$ is PI-equivalent to $E_{(-1, 1)}^{(\infty, \infty)}$ or
     \item The $\mathbb{Z}$-grading $E_{(m, n)}^{(\infty, \infty)}$ is PI-equivalent to $E_{(-1, 2)}^{(\infty, \infty)}$
     \end{enumerate}
     \end{theorem}

     
     A natural question would be: are such structures $\mathbb{Z}$-isomorphic? 
     Now we are going to investigate such question, given a negative answer. 
     \begin{theorem}\label{not isom geral}
     	Consider $m_{1}<0<n_{1}$ and $m_{2}<0<n_{2}$. The $\mathbb{Z}$-gradings $E_{(m_{1}, n_{1})}^{(\infty, \infty)}$ and $E_{(m_{2}, n_{2})}^{(\infty, \infty)}$ are  $\mathbb{Z}$-isomorphic if and only if $m_{1}=m_{2}$ and $n_{1}=n_{2}$.
     \end{theorem}
 \begin{proof}
 	It is enough to prove one part. Let $E_{(m_{1}, n_{1})}^{(\infty, \infty)}=\bigoplus_{r\in\mathbb{Z}}A_{r}$ and $E_{(m_{2}, n_{2})}^{(\infty, \infty)}=\bigoplus_{r\in\mathbb{Z}}B_{r}$. 
 	
 	Suppose that $\phi: E_{(m_{1}, n_{1})}^{(\infty, \infty)}\longrightarrow E_{(m_{2}, n_{2})}^{(\infty, \infty)}$ is a $\mathbb{Z}$-graded isomorphism. Hence we have that $\phi(A_{m_2})=B_{m_2}$. 
 	
 	As $\phi$ is a homomorphism we note that $\phi(e_{k})$ has not scalar part (because $\phi(e_{k})^{2}=0$ ), for all $k\in\mathbb{N}$. In this case, if $w\in B_E$ then
 	\[\phi(w)=\sum_{u\in B_{E}, |u|\geq |w| }\lambda_{u}u.\]
 	
 	Since $B_{m_2}$ has monomials of length one, then the same holds for $A_{m_2}$. Due to $L\subset A_{m_1}\cup A_{n_1}$, it follows that $m_{2}= m_{1}$ or $m_{2}= n_{1}$. As $m_{2}$ is a negative number we have that $m_{2}= m_{1}$. Similarly we show that $n_{2}= n_{1}$. 
 	\end{proof}
    
    The last result does not depend on the ground field $F$. Moreover we have a natural generalization of it in a more general context of $G$-gradings, as we can see below. 
    
    Given any group $G$ and elements $g_{1}, \ldots, g_{r}, h_{1}, \ldots, h_{s}$ in $G$, such that $g_{i}g_{j}=g_{j}g_{i}$ and $h_{i'}h_{j'}=h_{j'}h_{i'}$, for all $i, j\in\{1,\ldots, r\}$, and $i', j'\in\{1,\ldots, s\}$, consider the $G$-gradings $E_{(g_{1},\ldots,g_r)}^{(\infty,\ldots, \infty)}$ and $E_{(h_{1},\ldots,h_s)}^{(\infty,\ldots, \infty)}$ as presented in Remark \ref{induzindo em qualquer G}.

     \begin{proposition}
    	The $G$-gradings $E_{(g_{1},\ldots,g_r)}^{(\infty,\ldots, \infty)}$ and $E_{(h_{1},\ldots,h_s)}^{(\infty,\ldots, \infty)}$ are $G$-isomorphic if and only if $r=s$ and  $\{g_{1},\ldots,g_r\}=\{h_{1},\ldots,h_s\}$. 
    \end{proposition}
    \begin{proof}
    	Let $E_{(g_{1},\ldots,g_r)}^{(\infty,\ldots, \infty)}=A=\bigoplus_{g\in G}A_{g}$ and $E_{(h_{1},\ldots,h_s)}^{(\infty,\ldots, \infty)}=B=\bigoplus_{g\in G}B_{g}$. Suppose that $\phi: A\longrightarrow B$ is a $G$-graded isomorphism. Hence we have that $\phi(A_{h_1})=B_{h_1}$. 
    	
    	As in Theorem \ref{not isom geral}, $B_{h_1}$ has monomials of length one, then $A_{h_1}$ has also monomials of length one. Due to $L\subset A_{g_1}\cup \cdots \cup A_{g_r}$, it follows that there exists some $i$ such that $h_{1}= g_{i}$. Hence we can suppose without loss of generality that $i=1$. Similarly we show that $h_j$ belongs to $\{g_1,\ldots, g_r\}$, for each $j=2,\ldots, s$. We conclude that $s\leq r$ and $\{g_{1},\ldots,g_r\}\supseteq \{h_{1},\ldots,h_s\}$. The reverse inequality and the reverse inclusion are obtained  substituting $\phi$ by $\phi^{-1}$. 
    \end{proof}
    

    Section \ref{start} implies that the identities for  $E_{(m, n)}^{(\infty, \infty)}$ are completely determine by the parity of $m\times n$, while Theorem \ref{not isom geral} implies that both lower indexes $m$ and $n$ completely determine the structure of $E_{(m, n)}^{(\infty, \infty)}$ as $\mathbb{Z}$-graded algebra. 
    
    The structures given in Theorem \ref{pi equiv here} and Theorem \ref{not isom geral} are the first example of graded algebras whose support is infinite that are PI-equivalent and not isomorphic as graded algebras.

      \section{Central $\mathbb{Z}$-gradings on $E$}
      
      In this section we present results related to more general $\mathbb{Z}$-gradings on the Grassmann algebra, the so-called central $\mathbb{Z}$-gradings. Such structures have the strong property that each homogeneous component of even degree contains a ``significant part'' of the centre of $E$, which allows us to obtain its graded polynomial identities. In all $2$-induced $\mathbb{Z}$-grading of full support on $E$ studied  in Section \ref{start} it has been proved that each homogeneous component of even degree admits infinitely many monomials of even length with pairwise disjoint supports. This is the motivation to define a central $\mathbb{Z}$-grading on $E$.

       
       In the definition below we assume that the   $\mathbb{Z}$-grading on $E$ is homogeneous and of full support.
       
      \begin{definition}
      	Let $H=2\mathbb{Z}$. A $\mathbb{Z}$-grading $E=\bigoplus_{r\in\mathbb{Z}}A_{r}$ on $E$ is called central $\mathbb{Z}$-grading if each component $A_h$ has infinitely many elements of $B_E$ of even length with pairwise disjoint supports, for all $h\in H$. 
      	In this case, we write $E=E^{c}$ to indicate that $E$ is endowed with a central $\mathbb{Z}$-grading.
      \end{definition}
      \begin{example}
      	All $2$-induced $\mathbb{Z}$-gradings of full support are central $\mathbb{Z}$-gradings on $E$. The structure $E^{can}$ is not a central $\mathbb{Z}$-grading. Each component of even degree in $E_{(-2, 2)}^{(\infty, \infty)}$ has infinitely many elements of $B_E$ of even length with pairwise disjoint supports, but $E_{(-2, 2)}^{(\infty, \infty)}$ is not a central grading. 
      \end{example}
      
      	As commented in Section \ref{Preliminares}, a $\mathbb{Z}$-grading on $E$ induces naturally a $\mathbb{Z}_{2}$-grading on $E$. Next it will be proved that the $\mathbb{Z}$-graded identities of a central $\mathbb{Z}$-grading are closely related to the $\mathbb{Z}_{2}$-graded identities of its induced $\mathbb{Z}_{2}$-grading.
      
      \begin{example}
      	The grading $E_{(-1, 1)}^{(\infty, \infty)}$ is central, which induces the $\mathbb{Z}_{2}$-grading $E_{can}$.
      \end{example}
      
      Let $\pi\colon F\langle X|\mathbb{Z}\rangle\to F\langle X|\mathbb{Z}_{2}\rangle$ be the homomorphism of free algebras given by: 
      \[\pi(x_{i}^{n})=x_{i}^{\overline{n}}.\] 
      Here $x_{i}^{n}$ denotes a variable of $\mathbb{Z}$-degree $n$ while $x_{i}^{\overline{n}}$ is a variable of $\mathbb{Z}_{2}$-degree $\overline{n}$. For example,
       \[\pi([x_1^{-4}, x_{2}^{5}])=[x_1^{\overline{0}}, x_{2}^{\overline{1}}],\]
      \[\pi(x_1^{8}x_{2}^{8})=x_1^{\overline{0}}x_2^{\overline{0}}\] and so on.
      
      \begin{lema}
      	Let $I$ be a $T_{2}$-ideal of $F\langle X|\mathbb{Z}_{2}\rangle$ and $J=\{f\in F\langle X|\mathbb{Z}\rangle\mid \pi(f)\in I\}$. Then $J$ is a $T_{\mathbb{Z}}$-ideal of $F\langle X|\mathbb{Z}\rangle$.
      \end{lema}
      
      Given a central $\mathbb{Z}$-grading $E^{c}=\bigoplus_{r\in\mathbb{Z}}A_{r}$ on $E$, we denote by $E_{d}$ its induced $\mathbb{Z}_{2}$-grading. Recall that there exist three possibilities for $E_{d}$, namely: $E_{\infty}$, $E_{k^\ast}$ and $E_{k}$, for some non-negative integer $k$.
      
      \begin{lema}
      	Let $A=\bigoplus_{r\in\mathbb{Z}}A_{r}$ be an arbitrary $\mathbb{Z}$-grading on $E$ and $E_{d}=A_{\overline{0}}\oplus A_{\overline{1}}$ its induced $\mathbb{Z}_{2}$-grading. If $\pi(f)\in  T_{2}(E_{d})$ then $f\in T_{\mathbb{Z}}(A)$.
      \end{lema}
      \begin{proof}
      	Suppose that $\pi(f(x_{1}^{n_1},\ldots, x_{l}^{n_l}))=f(x_{1}^{\overline{n_1}},\ldots, x_{l}^{\overline{n_l}})\in  T_{2}(E_{d})$. As $A_{n}\subset A_{\overline{n}}$ for all $n\in\mathbb{Z}$, it is obtained immediately that $f(x_{1}^{n_1},\ldots, x_{l}^{n_l})\in T_{\mathbb{Z}}(A)$. 
      \end{proof}
      	
      In fact, the last lemma is more general. Let $G$ be an abelian group, let $H$ be a subgroup of $G$ and $f\in F\langle X|G\rangle $. 
      If $A$ is any $G$-graded algebra and $\pi(f)\in T_{G/H}(A)$ then $f\in T_{G}(A)$.

      The next result is an adaptation of \cite[Proposition 5.2]{centroneG1}. It will be important for our goals in this section. 
      \begin{theorem}\label{MAIN THEO}
      	Let $f\in F\langle X|\mathbb{Z}\rangle$ be a multilinear polynomial, suppose that $E^{c}=\bigoplus_{r\in\mathbb{Z}}A_{r}$ is a central $\mathbb{Z}$-grading on $E$ and let $E_{d}$ be its induced $\mathbb{Z}_{2}$-grading. Then $f\in T_{\mathbb{Z}}(E^{c})$ if and only if $\pi(f)\in  T_{2}(E_{d})$.
      \end{theorem}
      \begin{proof}
      	We just need to prove one part.
      	Let
      	
      	\[f(x_{1}^{n_1},\ldots, x_{l_{n_1}}^{n_1},\ldots, x_{p+1}^{n_r},\ldots, x_{p+l_{n_r}}^{n_r})\in T_{\mathbb{Z}}(E^{c}),\]
      	where $p=l_{n_1}+l_{n_2}+\cdots + l_{n_{r-1}}$.
      	
      	Let $H=2\mathbb{Z}$, $g=\pi(f)$ and consider an arbitrary $\mathbb{Z}_{2}$-graded substitution 
      	\[x_{j}^{\overline{n}} \longmapsto\sum_{m\in H} a_{j}^{n+m},\]
      	where each $a_j$ is an element of $B_E$ and $a^t$ means that $a$ is homogeneous of $\mathbb{Z}$-degree $t$. Since $f$ is  multilinear, for all $j$, we may assume that the substitution is of the form 
      	\[x_{j}^{\overline{n}} \longmapsto a_{j}^{n+m_{j}},\]
      	for some $m_{j}\in H$.
      	
      	If $m\in H$, the component $A_{m}$ has infinitely many monomials of even length with pairwise disjoint supports. Therefore, for each index $j$, there exists $b_{j}^{-m_{j}}$ of even length with $\mathbb{Z}$-degree $-m_j$ such that $a_{j}^{n+m_{j}}$ and $b_{j}^{-m_{j}}$ have pairwise disjoint supports. We can choose the monomials $b_{j}^{-m_{j}}$ with pairwise disjoint support. Then we define 
      	\[c_{j}^{n}=a_{j}^{n+m_{j}}b_{j}^{-m_{j}},\]
      	which is a homogeneous element of degree $n$ in the $\mathbb{Z}$-grading $E^{c}$.
      	
      	Now we consider the $\mathbb{Z}$-graded substitution given by
      	\[x_{j}^{n} \longmapsto c_{j}^{n}.\]
      	
      	Due to $f\in T_{\mathbb{Z}}(E^{c})$ and each $b_{j}^{-m_j}\in Z(E)$ we obtain
      	\[0=f(c_{1}^{n_1},\ldots, c_{p+l_{n_r}}^{n_r})=f(a_{1}^{n_{1}+m_{1}}b_{1}^{-m_{1}},\ldots,a_{p+l_{n_r}}^{n_{r}+m_{p+l_{n_r}}}b_{p+l_{n_r}}^{-m_{p+l_{n_r}}} )=\]
      	
      	\[=(\prod_{j=1}^{p+l_{n_r}}b_{j}^{-m_j})g(a_{1}^{\overline{n_1}},\ldots, a_{p+l_{n_r}}^{\overline{n_r}}).\]
      	
      	Since the monomials $b_j$'s have pairwise disjoint supports, we obtain
      	\[0=g(a_{1}^{\overline{n_1}},\ldots, a_{p+l_{n_r}}^{\overline{n_r}})\]
      	and $\pi(f)=g\in T_{\mathbb{Z}}(E_{d})$.
      \end{proof}
      
      As an immediate consequence of Theorem \ref{MAIN THEO} we obtain a description of the $\mathbb{Z}$-graded identities for all central $\mathbb{Z}$-grading on $E$, namely:
      
      \begin{corollary}\label{identidades Z das identidades Z2}
      	Let $E^{c}=\bigoplus_{r\in\mathbb{Z}}A_{r}$ be a central $\mathbb{Z}$-grading on $E$ and let $E_{d}=A_{\overline{0}}\oplus A_{\overline{1}}$ be its induced $\mathbb{Z}_2$-grading. Let $I=T_{2}(E_{d})$ and $J=\{f\in F\langle X|\mathbb{Z}\rangle\mid \pi(f)\in I\}$. Then we have that $T_{\mathbb{Z}}(E^{c})= J$.

      \end{corollary}
      \begin{proof}
      	It is an immediate application of Theorem \ref{MAIN THEO}. 
      \end{proof}
      
       In Section \ref{start} we proved that the graded identities for a grading of type $E_{(m, n)}^{(\infty, \infty)}$ of full support are completely determined by the parity of $m\times n$. Another way to describe such identities is looking at the induced $\mathbb{Z}_{2}$-grading. 
      
      \begin{corollary}\label{cor}
      	Let $m<0<n$ be integer numbers, suppose that $E_{(m, n)}^{(\infty, \infty)}$ is of full support and let $E_{d}$ be its induced $\mathbb{Z}_{2}$-grading. Let $I=T_{2}(E_{d})$ and $J=\{f\in F\langle X|\mathbb{Z}\rangle\mid \pi(f)\in I\}$. Then $T_{\mathbb{Z}}(E_{(m, n)}^{(\infty, \infty)})= J$
       \end{corollary}
    
      \begin{example}
      	As $[x_{1}, x_{2}]$ (when $\alpha(x_{1})=\overline{0}$ or $\alpha(x_{2})=\overline{0}$) and $y_{1}y_{2}+y_{2}y_{1}$ (when $\alpha(y_{1})=\alpha(y_{2})=\overline{1}$) generate the $\mathbb{Z}$-graded identities for $E_{can}$ in characteristic zero and $E_{(-1, 1)}^{(\infty, \infty)}$ is central, it follows that the same polynomials generate the $\mathbb{Z}$-graded identities for $E_{(-1, 1)}^{(\infty, \infty)}$ when we exchange the variables of $\mathbb{Z}_2$-degree $\overline{0}$ for variables of $\mathbb{Z}$-degree even and  the variables of $\mathbb{Z}_2$-degree $\overline{1}$ for variables of $\mathbb{Z}$-degree odd. Such polynomials are exactly the polynomials that were used in Section \ref{start} to define the variety $\mathfrak{V}_{2}^{\mathbb{Z}}$. 
      \end{example}
      \begin{example}\label{3 indices}
      	Given $k$ a  non-negative integer number, let $I_{1}$ and $I_{2}$ be infinite and disjoint sets of integer numbers such that
      	$I_{1}\cup I_{2}=\{k+1, k+2,\ldots\}$.

      	Now we define a $\mathbb{Z}$-grading on $E$ as follows:
      	
      	\begin{align*}
      	\|e_{i}\|^{\tilde{k}} & =\begin{cases} 0,\text{ if }  i=1,\ldots,k\\ 1, \text{ if } i\in I_{1}\\
      	-1 \text{ if } i\in I_{2} 
      	\end{cases}.
      	\end{align*}
      	  
      	 We extend the $\mathbb{Z}$-degree for all element in $B_E$ as usual. Let $E^{\tilde{k}}=\bigoplus_{r\in\mathbb{Z}}A_{r}$ be such $\mathbb{Z}$-grading on $E$. 
      	 
      	 It is easy to check that:
      	 \begin{itemize}
      	 	\item If $B$ is the subalgebra of $E^{\tilde{k}}$ generated by all $e_i$ where $i\in \{1,\ldots, k\}\cup I_1$ then $B$ is homogeneous and it is $\mathbb{Z}$-isomorphic to $E^{k}$. 
      	 	\item The $\mathbb{Z}$-grading $E^{\tilde{k}}$ induces $E_{k}$ by a quocient grading module $2\mathbb{Z}$.
      	 	\item The $\mathbb{Z}$-grading $E^{\tilde{k}}$ is central.
      	 \end{itemize}
      	 
      	 Then the $\mathbb{Z}$-graded identities of $E^{\tilde{k}}$ can be raised of the $\mathbb{Z}_2$-graded identities of $E_{k}$.

       \end{example}
       
       \begin{example}
       	Given $k$ a  non-negative integer number, it is known that $E^k$ induces $E_k$. If $r<0$ and $A_{r}$ is the respective component of $E^k$, we have that $A_{r}=0$. Hence $x$ is a graded identity for $E^k$, if $\alpha(x)=r$. On the other hand, $\pi(x)$ is not a $\mathbb{Z}_2$-graded identity for $E_k$. Therefore, it is not possible to raise the $\mathbb{Z}$-graded identity of $E^k$ using the identities of $E_{k}$. This happens because $E^k$ is not a central grading. 
       \end{example}

       Given an infinity subset $L'$ of $\{e_{1}, e_{2}, \ldots\}$, the subalgebra $E'$ of $E$ generated by $1$ and $L'$ is naturally isomorphic to $E$. So it makes sense to talk about central $\mathbb{Z}$-gradings on $E'$ as well. This will be used in the next theorem.

       \begin{theorem}\label{teo não homogeneo}
       Let $A=\bigoplus_{r\in\mathbb{Z}}A_{r}$ be an arbitrary $\mathbb{Z}$-grading on $E$ (not necessarily homogeneous). Suppose that there exist infinitely many elements $e_{i}\in L$ homogeneous in the grading which generate a subalgebra $B$ endowed with a central $\mathbb{Z}$-grading. Let $E_{d}$ be the $\mathbb{Z}_{2}$-grading induced  by $B$, $I=T_{2}(E_{d})$ and $J=\{f\in F\langle X|\mathbb{Z}\rangle\mid \pi(f)\in I\}$. In these conditions we have that $T_{\mathbb{Z}}(A)\subset J$. In particular, if there exist $m<0<n$ such that $\gcd(m, n)=1$ and both $L\cap A_{m}$ and $L\cap A_{n}$ are infinite, then
       \begin{enumerate}
       	\item [(a)] $T_{\mathbb{Z}}(A)=T_{\mathbb{Z}}(\mathfrak{V}_{1}^\mathbb{Z})$ or
       	\item [(b)] $T_{\mathbb{Z}}(A)\subset T_{\mathbb{Z}}(\mathfrak{V}_{2}^\mathbb{Z})$

       \end{enumerate}
   \end{theorem} 
       \begin{proof}
       	As $B$ is a central $\mathbb{Z}$-grading, according to Corollary \ref{identidades Z das identidades Z2}, we have that $T_{\mathbb{Z}}(B)=J$. Since $B$ is a homogeneous subalgebra of $A$, it follows that $T_{\mathbb{Z}}(A)\subset T_{\mathbb{Z}}(B)=J$ and the first part follows. 
       	
       	Now assume that there exist $m<0<n$ such that $\gcd(m, n)=1$ and both $L\cap A_{m}$ and $L\cap A_{n}$ are infinite sets. We define $B$ as the subalgebra of $A$ generated by all $e_{i}$ in $L\cap A_{m}$ or $L\cap A_{n}$. Hence $B$ is a homogeneous subalgebra of $A$ and $\mathbb{Z}$-isomorphic to some of the structures given in Section \ref{start}. If $B$ generates $\mathfrak{V}_{1}^\mathbb{Z}$, statement $(a)$ holds. If $B$ generates $\mathfrak{V}_{2}^\mathbb{Z}$, then $T_{\mathbb{Z}}(A)\subset T_{\mathbb{Z}}(\mathfrak{V}_{2}^\mathbb{Z})$ and $(b)$ holds.
       \end{proof} 
      
   Similar to Theorem \ref{teo não homogeneo}, we have the following result. 
   \begin{theorem}
   	Let $A=\bigoplus_{r\in\mathbb{Z}}A_{r}$ be an arbitrary $\mathbb{Z}$-grading on $E$ (not necessarily homogeneous). Suppose that there exist integers $m<0<n$, with opposite parities, such that both $L\cap A_{m}$ and $L\cap A_{n}$ are infinite sets. Then 
   	$T_{\mathbb{Z}}(A)=T_{\mathbb{Z}}(\mathfrak{V}_{1}^\mathbb{Z}),$
   	up to graded monomial identities of degree 1.
   \end{theorem}
   \begin{proof}
   	It is clear that $[x_1,x_2,x_3]$ is a graded identity for $A$, for every choice of degree $\alpha(x_1)$, $\alpha(x_2)$, $\alpha(x_3)$. As in Theorem \ref{teo não homogeneo}, we define $B$ as being the subalgebra of $A$ generated by $(L\cap A_{m})\cup (L\cap A_{n})$. In this case $B$ is $\mathbb{Z}$-isomorphic to $E_{(m, n)}^{(\infty, \infty)}$, which is not necessarily of full support. By Theorem \ref{corbom}, it follows that $Sup(B)=d\mathbb{Z}$, where $d=\gcd(m,n)$. According to Theorem \ref{idquasefull} the ideal $T_{\mathbb{Z}}(E_{(m, n)}^{(\infty, \infty)})$ is generated by the generators of $T_{\mathbb{Z}}(E_{(\frac{m}{d}, \frac{n}{d})}^{(\infty, \infty)})$ together with
   	$$N=\{x\in F\langle X|\mathbb{Z}\rangle\mid \alpha(x)\notin d\mathbb{Z}\}.$$
   	We have that $E_{(\frac{m}{d}, \frac{n}{d})}^{(\infty, \infty)}$ is of full support. As $m$ and $n$ have opposite parities, it follows that  $E_{(\frac{m}{d}, \frac{n}{d})}^{(\infty, \infty)}$ induces the grading $E_\infty$. Hence $T_{\mathbb{Z}}(E_{(\frac{m}{d}, \frac{n}{d})}^{(\infty, \infty)})=T_{\mathbb{Z}}(\mathfrak{V}_{1}^\mathbb{Z})$. This implies that
   	
   	$$T_{\mathbb{Z}}(\mathfrak{V}_{1}^\mathbb{Z})\subseteq  T_{\mathbb{Z}}(A)\subseteq T_{\mathbb{Z}}(B),$$
   	and therefore the result follows, since $T_{\mathbb{Z}}(B)$ differs from $T_{\mathbb{Z}}(\mathfrak{V}_{1}^\mathbb{Z})$ by the graded monomial identities of degree 1.
   \end{proof}
   	

  \section{Further discussion}
  Throughout this section some possible generalizations of the results presented in this paper will be discussed. Here our intention is to motivate future research on the subject.
  
  \begin{itemize}
  	\item \textbf{For other fields} 
  \end{itemize}

  Let $F$ be an infinite field of characteristic $p>2$. Consider the variety $\mathfrak{V}_{1}^\mathbb{Z}$ of $\mathbb{Z}$-graded algebras define by
  \begin{enumerate} 
  \item $[x_{1}, x_{2}, x_{3}]$, for all degree of $x_{1}$, $x_{2}$, $x_{3}$.
  \item $x^p$, for all $x$ with degree different from zero.
  \end{enumerate} 
  Moreover consider the variety $\mathfrak{V}_{2}^\mathbb{Z}$ of $\mathbb{Z}$-graded algebras define by
  
  \begin{enumerate}
  \item $[x_{1}, x_{2}]$, if either $\alpha(x_{1})$ or $\alpha(x_{2})$ is even.
  
  \item $x_{1}x_{2}+x_{2}x_{1}$, if both $\alpha(x_{1})$ and $\alpha(x_{2})$ are odd.
  
  \item $x^{p}$, if $\alpha(x)>0$ is even. 
  \end{enumerate} 
  
  Consider the Grassmann algebra $E$ over $F$. We believe that the $\mathbb{Z}$-gradings studied in Section \ref{start} generate either the variety $\mathfrak{V}_{1}^\mathbb{Z}$ or the variety $\mathfrak{V}_{2}^\mathbb{Z}$. However the arguments and notations become rather clumby. This is why we develop the paper considering $F$ of characteristic zero. 
  
  \begin{itemize}
  	\item \textbf{For other groups} 
  \end{itemize}
  Let $G$ be any group  (not necessarily abelian), and $g_{1}, \ldots, g_{r}$ in $G$ such that $g_{i}g_{j}=g_{j}g_{i}$, for all $i, j\in\{1,\ldots, r\}$. As commented in Remark \ref{induzindo em qualquer G}, using $g_{1}, \ldots, g_{r}$, we can induce a $G$-grading on $E$. Hence one can investigate conditions for full support in this case. 
  \begin{theorem}
  	Let $G$ be an abelian group. There exists a $r$-induced $G$-grading on $E$ of full support if and only if the group $G$ is finitely generated. 
  \end{theorem}
  \begin{proof}
  	Let us consider $G$ with additive notation. Assuming that $G$ is finitely generated, consider $g_{1},\ldots, g_{n}\in G$ such that $G=\langle g_{1},\ldots, g_{n}\rangle$ and the decomposition of $L$:
  	\[L=L_{g_{1}}^{\infty}\oplus\cdots\oplus L_{g_{n}}^{\infty}\oplus L_{g_{1}^{-1}}^{\infty}\oplus\cdots\oplus L_{g_{n}^{-1}}^{\infty}.\]
  	Let $E_{(g_{1},\ldots, g_{n},g_{1}^{-1},\ldots, g_{n}^{-1})}^{(\infty,\ldots, \infty, \infty,\ldots, \infty)}$ be the induced $G$-grading on $E$. Given $g\in G$, we have that $g=a_{1}+\cdots + a_{m}$, where $a_{i}\in \{g_{1},\ldots, g_{n},g_{1}^{-1},\ldots, g_{n}^{-1}\},$ and $m$ is a positive integer. For each $k\in \{1,\ldots, m\}$, we take $j_{k}\in\mathbb{N}$ such that $e_{j_k}\in L_{a_k}^{\infty}$. In this case, 
  	\[\|e_{j_1}\cdots e_{j_m}\|=\|e_{j_1}\|+\cdots + \|e_{j_m}\|=a_{1}+\cdots + a_{m}=g,\]
  	and $e_{j_1}\cdots e_{j_m}\neq 0$. Hence $g\in Sup(E_{(g_{1},\ldots, g_{n},g_{1}^{-1},\ldots, g_{n}^{-1})}^{(\infty,\ldots, \infty, \infty,\ldots, \infty)})$.
  	
  	Now suppose that there exist $g_{1},\ldots, g_{r}\in G$ and $v_{1},\ldots, v_{r}\in\mathbb{N}\cup \{\infty\}$ such that $E_{(g_{1},\ldots, g_{r})}^{(v_{1},\ldots,v_{r})}$ is of full support. In this case, given $g\in G$ there exist $i_{1},\ldots, i_{m}, j_{1},\ldots, j_{m}\in\mathbb{N}$ such that $e_{j_k}\in L_{g_{i_k}}^{v_{i_k}}$, for all $k\in \{1, \ldots, m\}$, and $0\neq e_{j_1}\cdots e_{j_m}\in A_{g}$. Therefore,
  	\[g=\|e_{j_1}\|+\cdots +\|e_{j_m}\|=g_{i_1}+\cdots + g_{i_m}\]
  	and we obtain $g\in \langle g_{1},\ldots, g_{r}\rangle$. Hence we have that $G=\langle g_{1},\ldots, g_{r}\rangle$ and we are done.
  \end{proof}

  Such result is an initial motivation to address the problem of the gradings on $E$ by infinite groups other than $\mathbb{Z}$.

  \begin{itemize}
  	\item \textbf{For $3$-induced $\mathbb{Z}$-gradings} 
  \end{itemize}
  In view of the results of Section \ref{start} a natural forwarding is to study $r$-induced $\mathbb{Z}$-gradings on $E$ of full support, when $r>2$. This seems to be an interesting problem. Recall that the Example \ref{3 indices} presents a kind of $3$-induced $\mathbb{Z}$-grading on $E$ of full support whose $T_{\mathbb{Z}}$-ideal of graded identities is different from both $T_{\mathbb{Z}}(\mathfrak{V}_{1}^\mathbb{Z})$ and $T_{\mathbb{Z}}(\mathfrak{V}_{2}^\mathbb{Z})$.


\begin{flushleft}
	\textbf{Acknowledgments}
\end{flushleft}
The third author acknowledges the support by the FAPESP grant No. 2019/12498-0.

\end{document}